\documentclass{amsart}
\usepackage{amsmath}
\usepackage{amssymb}
\usepackage{amsthm}
\usepackage{epsfig}
\usepackage{amsfonts}
\usepackage{amscd}
\usepackage{mathrsfs}
\usepackage{enumerate}
\usepackage{latexsym}
\usepackage{url}

\newtheorem{theorem}{Theorem}[section]
\newtheorem{lemma}[theorem]{Lemma}
\newtheorem{proposition}[theorem]{Proposition}

\theoremstyle{definition}
\newtheorem{definition}[theorem]{Definition}
\newtheorem{example}[theorem]{Example}

\theoremstyle{remark}
\newtheorem{remark}[theorem]{Remark}
\numberwithin{equation}{section}

\begin{document}

\title[Representations of certain normed algebras]{Representations of certain normed algebras}

\author{M.R. Koushesh}
\address{Department of Mathematical Sciences, Isfahan University of Technology, Isfahan 84156--83111, Iran}
\address{School of Mathematics, Institute for Research in Fundamental Sciences (IPM), P.O. Box: 19395--5746, Tehran, Iran}
\email{koushesh@cc.iut.ac.ir}
\thanks{This research was in part supported by a grant from IPM (No. 90030052).}

\subjclass[2010]{Primary 54D35, 46J10, 46J25, 46E25, 46E15; Secondary 54C35, 46H05, 16S60}

%\date{October 17, 2011 and, in revised form, .}

\keywords{Stone--\v{C}ech compactification; Commutative Gelfand--Naimark Theorem; Real Banach algebra; Gelfand theory; Compact support; Lindel\"{o}f; Hewitt realcompactification; Countably compact.}

\begin{abstract}
We show that for a normal locally-${\mathscr P}$ space $X$ (where ${\mathscr P}$ is a topological property subject to some mild requirements) the subset $C_{\mathscr P}(X)$ of $C_b(X)$ consisting of those elements whose support has a neighborhood with ${\mathscr P}$, is a subalgebra of $C_b(X)$ isometrically isomorphic to $C_c(Y)$ for some unique (up to homeomorphism) locally compact Hausdorff space $Y$. The space $Y$ is explicitly constructed as a subspace of the Stone--\v{C}ech compactification $\beta X$ of $X$ and contains $X$ as a dense subspace. Under certain conditions, $C_{\mathscr P}(X)$ coincides with the set of those elements of $C_b(X)$ whose support has ${\mathscr P}$, it moreover becomes a Banach algebra, and simultaneously, $Y$ satisfies $C_c(Y)=C_0(Y)$. This includes the cases when ${\mathscr P}$ is the Lindel\"{o}f property and $X$ is either a locally compact paracompact space or a locally-${\mathscr P}$ metrizable space. In either of the latter cases, if $X$ is non-${\mathscr P}$, $Y$ is non-normal, and $C_{\mathscr P}(X)$ fits properly between $C_0(X)$ and $C_b(X)$; even more, we can fit a chain of ideals of certain length between $C_0(X)$ and $C_b(X)$. The known construction of $Y$ enables us to derive a few further properties of either $C_{\mathscr P}(X)$ or $Y$. Specifically, when ${\mathscr P}$ is the Lindel\"{o}f property and $X$ is a locally-${\mathscr P}$ metrizable space, we show that
\[\dim C_{\mathscr P}(X)=\ell(X)^{\aleph_0},\]
where $\ell(X)$ is the Lindel\"{o}f number of $X$, and when ${\mathscr P}$ is countable compactness and $X$ is a normal space, we show that
\[Y=\mathrm{int}_{\beta X}\upsilon X\]
where $\upsilon X$ is the Hewitt realcompactification of $X$.
\end{abstract}

\maketitle

%\tableofcontents

\section{Introduction}

Throughout this article the underlying field of scalars (which is fixed throughout each discussion) is assumed to be either the real field $\mathbb{R}$ or the complex field $\mathbb{C}$, unless specifically stated otherwise. Also, we will use the term {\em space} to refer only to a topological space; we may assume that spaces are non-empty.

Let $X$ be a space. Denote by $C_b(X)$ the set of all continuous bounded scalar-valued functions on $X$. If $f\in C_b(X)$, the {\em zero-set} of $f$, denoted by $\mathrm{Z}(f)$, is $f^{-1}(0)$, the {\em cozero-set} of $f$, denoted by $\mathrm{Coz}(f)$, is $X\backslash\mathrm{Z}(f)$, and the {\em support} of $f$, denoted by $\mathrm{supp}(f)$, is $\mathrm{cl}_X\mathrm{Coz}(f)$. Let
\[\mathrm{Z}(X)=\big\{\mathrm{Z}(f):f\in C_b(X)\big\}\]
and
\[\mathrm{Coz}(X)=\big\{\mathrm{Coz}(f):f\in C_b(X)\big\}.\]
Denote by $C_0(X)$ the set of all $f\in C_b(X)$ which vanish at infinity (that is, $|f|^{-1}([\epsilon,\infty))$ is compact for each $\epsilon>0$) and denote by $C_c(X)$ the set of all $f\in C_b(X)$ with compact support.

Let ${\mathscr P}$ be a topological property. Then
\begin{itemize}
  \item ${\mathscr P}$ is {\em closed} ({\em open}, respectively) {\em hereditary}, if any closed (open, respectively) subspace of a space with ${\mathscr P}$, also has ${\mathscr P}$.
  \item ${\mathscr P}$ is {\em preserved under finite} ({\em countable}, {\em locally finite} respectively) {\em closed sums}, if any space which is expressible as a finite (countable, locally finite, respectively) union of its closed subspaces each having ${\mathscr P}$, also has ${\mathscr P}$.
\end{itemize}

Let $X$ be a space and let ${\mathscr P}$ be a topological property. The space $X$ is called a {\em ${\mathscr P}$-space} if it has ${\mathscr P}$. A {\em ${\mathscr P}$-subspace} of $X$ is a subspace of $X$ which has ${\mathscr P}$. By a {\em ${\mathscr P}$-neighborhood} of a point (set, respectively) in $X$  we mean a neighborhood of the point (set, respectively) in $X$ having ${\mathscr P}$. The space $X$ is called {\em locally-${\mathscr P}$} if each of its points has a ${\mathscr P}$-neighborhood in $X$. Note that if $X$ is regular and ${\mathscr P}$ is closed hereditary, then $X$ is locally-${\mathscr P}$ if and only if each $x\in X$ has an open neighborhood $U$ in $X$ such that $\mathrm{cl}_XU$ has ${\mathscr P}$.

For other undefined terms and notation we refer to the standard text \cite{E}. (In particular, compact and paracompact spaces are Hausdorff - thus locally compact spaces are completely regular - Lindel\"{o}f spaces are regular, etc.)

The normed subalgebra $C_c(X)$ of $C_b(X)$ consisting of those elements whose support is compact (if $X$ is locally compact, equivalently, consisting of those elements whose support has a compact neighborhood in $X$) is crucial. Here, motivated by our previous work \cite{Ko6} (in which we have studied the Banach algebra of continuous bounded scalar-valued functions with separable support on a locally separable metrizable space $X$) we replace compactness by a rather general topological property ${\mathscr P}$, thus, considering the subset $C_{\mathscr P}(X)$ of $C_b(X)$ consisting of those elements whose support has a ${\mathscr P}$-neighborhood in $X$. Obviously, $C_{\mathscr P}(X)$ is identical to $C_c(X)$ if ${\mathscr P}$ is compactness and $X$ is locally compact. We show that for a normal locally-${\mathscr P}$ space $X$ (with ${\mathscr P}$ subject to some mild requirements) $C_{\mathscr P}(X)$ is a subalgebra of $C_b(X)$ isometrically isomorphic to $C_c(Y)$ for some unique (up to homeomorphism) locally compact space $Y$. The space $Y$, which is explicitly constructed as a subspace of the Stone--\v{C}ech compactification $\beta X$ of $X$, contains $X$ as a dense subspace. Under certain conditions, $C_{\mathscr P}(X)$ coincides with the set of those elements of $C_b(X)$ whose support has ${\mathscr P}$, it becomes moreover a Banach algebra, and at the same time, $Y$ satisfies $C_c(Y)=C_0(Y)$; thus, in particular, in such cases $Y$ is countably compact. This includes the cases when ${\mathscr P}$ is the Lindel\"{o}f property and $X$ is either a locally compact paracompact space or a locally-${\mathscr P}$ metrizable space. In either of the latter cases, if $X$ is non-${\mathscr P}$, then $Y$ is non-normal, and $C_{\mathscr P}(X)$ fits properly between $C_0(X)$ and $C_b(X)$; even more, we can fit a chain of ideals of certain length between $C_0(X)$ and $C_b(X)$. (This shows how differently $C_{\mathscr P}(X)$ may behave for different topological properties ${\mathscr P}$: If ${\mathscr P}$ is compactness, then for any locally-${\mathscr P}$ metrizable space $X$, if $C_{\mathscr P}(X)$ is a Banach algebra then $X$ has ${\mathscr P}$, while, if ${\mathscr P}$ is the Lindel\"{o}f property, there exist some non-${\mathscr P}$ locally-${\mathscr P}$ metrizable spaces $X$ such that $C_{\mathscr P}(X)$ is a Banach algebra.) A few further properties of $Y$ or $C_{\mathscr P}(X)$ are also derived through the known construction of $Y$. Specifically, when ${\mathscr P}$ is the Lindel\"{o}f property and $X$ is a locally-${\mathscr P}$ metrizable space we show that
\[\dim C_{\mathscr P}(X)=\ell(X)^{\aleph_0},\]
where $\ell(X)$ is the Lindel\"{o}f number of $X$, and when ${\mathscr P}$ is countable compactness and $X$ is a normal space we show that
\[Y=\mathrm{int}_{\beta X}\upsilon X,\]
where $\upsilon X$ is the Hewitt realcompactification of $X$. Results of this article are further generalized in the follow-up manuscript \cite{Ko14} available on the arXiv. (See also \cite{Ko11}.)

Let $X$ be a completely regular space. In the recent arXiv preprint \cite{T}, for a filter base ${\mathscr B}$ of open subspaces of $X$, the author studies $C_{\mathscr B}(X)$ defined as the set of all $f\in C(X)$ with support contained in $X\backslash A$ for some $A\in{\mathscr B}$. Also, if ${\mathscr I}$ is an ideal of closed subspaces of $X$, in \cite{AG}, the authors consider $C_{\mathscr I}(X)$ defined as the set of all $f\in C(X)$ with support contained in ${\mathscr I}$. Our approach here is quite different from that of either \cite{T} or \cite{AG}. The interested reader may find it useful to compare our results with those obtained in \cite{T} and \cite{AG}.  (See also \cite{AN} for related results.)

We now review briefly some known facts from General Topology. Additional information on the subject may be found in \cite{E}, \cite{GJ} and \cite{PW}.

\subsection*{1.1. The Stone--\v{C}ech compactification.} Let $X$ be a completely regular space. The {\em Stone--\v{C}ech compactification} $\beta X$ of $X$ is the compactification of $X$ characterized among all compactifications of $X$ by the following property: Every continuous $f:X\rightarrow K$, where $K$ is a compact space, is continuously extendable over $\beta X$; denote by $f_\beta$ this continuous extension of $f$. The Stone--\v{C}ech compactification of a completely regular space always exists. Use will be made in what follows of the following properties of $\beta X$. (See Sections 3.5 and 3.6 of \cite{E}.)
\begin{itemize}
  \item $X$ is locally compact if and only if $X$ is open in $\beta X$.
  \item Any open-closed subspace of $X$ has open-closed closure in $\beta X$.
  \item If $X\subseteq T\subseteq\beta X$ then $\beta T=\beta X$.
  \item If $X$ is normal then $\beta T=\mathrm{cl}_{\beta X}T$ for any closed subspace $T$ of $X$.
\end{itemize}

\subsection*{1.2. The Hewitt realcompactification.} A space is called {\em realcompact} if it is homeomorphic to a closed subspace of some product $\mathbb{R}^\alpha$. Let $X$ be a completely regular space. A {\em realcompactification} of $X$ is a realcompact space containing $X$ as a dense subspace. The {\em Hewitt realcompactification} $\upsilon X$ of $X$ is the realcompactification of $X$ characterized among all realcompactifications of $X$ by the following property: Every continuous $f:X\rightarrow \mathbb{R}$ is continuously extendable over $\upsilon X$. One may assume that $\upsilon X\subseteq\beta X$.

\subsection*{1.3. Paracompact spaces and the Lindel\"{o}f property.} Let $X$ be a space. For open covers $\mathscr{U}$ and $\mathscr{V}$ of $X$ we say that $\mathscr{U}$ is a {\em refinement} of $\mathscr{V}$ (or $\mathscr{U}$ {\em refines} $\mathscr{V}$) if each element of $\mathscr{U}$ is contained in an element of $\mathscr{V}$. An open cover $\mathscr{U}$ of $X$ is called {\em locally finite} if each point of $X$ has a neighborhood in $X$ intersecting only a finite number of the elements of $\mathscr{U}$. The space $X$ is called {\em paracompact} if it is Hausdorff and for every open cover $\mathscr{U}$ of $X$ there exists a locally finite open cover of $X$ which refines $\mathscr{U}$. Every metrizable space and every Lindel\"{o}f space is paracompact and every paracompact space is normal. Any locally compact paracompact space $X$ can be represented as a disjoint union
\[X=\bigcup_{i\in I}X_i,\]
where $I$ is an index set, and $X_i$'s are Lindel\"{o}f open-closed subspaces of $X$. (See Theorem 5.1.27 of \cite{E}.)

\subsection*{1.4. Metrizable spaces and the Lindel\"{o}f property.} The {\em Lindel\"{o}f number} of a space $X$, denoted by $\ell(X)$, is defined by
\[\ell(X)=\min\{\mathfrak{n}:\mbox{any open cover of $X$ has a subcover of cardinality}\leq\mathfrak{n}\}+\aleph_0.\]
In particular, a space $X$ is Lindel\"{o}f if and only if $\ell(X)=\aleph_0$. By a theorem of Alexandroff, any locally Lindel\"{o}f metrizable space $X$ can be represented as a disjoint union
\[X=\bigcup_{i\in I}X_i,\]
where $I$ is an index set, and $X_i$'s are non-empty Lindel\"{o}f open-closed subspaces of $X$. (See Problem 4.4.F of \cite{E}; note that in metrizable spaces the two notions of separability and being Lindel\"{o}f coincide.) Observe that $\ell(X)=|I|$ if $I$ is infinite.

\section{The normed algebra $C_{\mathscr P} (X)$}

We begin our study by considering the general normed algebra $C_{\mathscr P}(X)$ as defined below.

\begin{definition}\label{HWA}
Let $X$ be a space and let ${\mathscr P}$ be a topological property. Define
\[C_{\mathscr P}(X)=\big\{f\in C_b(X):\mathrm{supp}(f)\mbox{ has a ${\mathscr P}$-neighborhood}\big\}.\]
\end{definition}

\begin{remark}
Note that
\[C_{\mathscr P}(X)=C_c(X)\]
if ${\mathscr P}$ is compactness and $X$ is a locally compact space; to see let $f\in C_{\mathscr P}(X)$. Then $\mathrm{supp}(f)$ is compact, as it has a compact neighborhood in $X$. For the converse, suppose that $\mathrm{supp}(g)$ is compact for some $g\in C_b(X)$. For each $x\in X$ let $U_x$ be an open neighborhood of $x$ in $X$ with compact closure $\mathrm{cl}_XU_x$. By compactness of $\mathrm{supp}(g)$ there exist $x_1,\ldots,x_n\in X$ such that
\[\mathrm{supp}(g)\subseteq U_{x_1}\cup\cdots\cup U_{x_n}=U.\]
Now
\[\mathrm{cl}_XU=\mathrm{cl}_XU_{x_1}\cup\cdots\cup\mathrm{cl}_X U_{x_n}\]
is a neighborhood of $\mathrm{supp}(g)$ in $X$ and it is compact, as it is a finite union of compact subspaces of $X$. Thus $g\in C_{\mathscr P}(X)$.
\end{remark}

The following subspace of $\beta X$, introduced in \cite{Ko3} (see also \cite{Ko4}, \cite{Ko13} and \cite{Ko12}), plays a crucial role in what follows.

\begin{definition}\label{RRA}
For a completely regular space $X$ and a topological property ${\mathscr P}$, let
\[\lambda_{\mathscr P} X=\bigcup\big\{\mathrm{int}_{\beta X}\mathrm{cl}_{\beta X}C:C\in\mathrm{Coz}(X)\mbox{ and }\mathrm{cl}_XC \mbox{ has }{\mathscr P}\big\},\]
considered as a subspace of $\beta X$.
\end{definition}

\begin{remark}
Note that in Definition \ref{RRA} we have
\[\lambda_{\mathscr P} X=\bigcup\big\{\mathrm{int}_{\beta X}\mathrm{cl}_{\beta X}Z:Z\in\mathrm{Z}(X)\mbox{ has }{\mathscr P}\big\},\]
provided that ${\mathscr P}$ is a closed hereditary topological property. (See \cite{Ko4}.)
\end{remark}

If $X$ is a space and $D$ is a dense subspace of $X$, then
\[\mathrm{cl}_XU=\mathrm{cl}_X(U\cap D)\]
for every open subspace $U$ of $X$. This will be used in the following simple observation.

\begin{lemma}\label{LKG}
Let $X$ be a completely regular space and let $f:X\rightarrow[0,1]$ be continuous. If $0<r<1$ then
\[f_\beta^{-1}\big[[0,r)\big]\subseteq\mathrm{int}_{\beta X}\mathrm{cl}_{\beta X}f^{-1}\big[[0,r)\big].\]
\end{lemma}

\begin{proof}
Note that
\[\mathrm{cl}_{\beta X}f_\beta^{-1}\big[[0,r)\big]=\mathrm{cl}_{\beta X}\big(X\cap f_\beta^{-1}\big[[0,r)\big]\big)=\mathrm{cl}_{\beta X}f^{-1}\big[[0,r)\big]\]
and that
\[f_\beta^{-1}\big[[0,r)\big]\subseteq\mathrm{int}_{\beta X}\mathrm{cl}_{\beta X}f_\beta^{-1}\big[[0,r)\big].\]
\end{proof}

The following is a slight modification of Lemma 2.10 of \cite{Ko3}.

\begin{lemma}\label{BBV}
Let $X$ be a completely regular locally-${\mathscr P}$ space, where ${\mathscr P}$ is a closed hereditary topological property. Then
\[X\subseteq\lambda_{\mathscr P} X.\]
\end{lemma}

\begin{proof}
Let $x\in X$ and let $U$ be an open neighborhood of $x$ in $X$ whose closure $\mathrm{cl}_XU$ has ${\mathscr P}$. Let $f:X\rightarrow[0,1]$ be continuous with
\[f(x)=0\;\;\;\;\mbox{ and }\;\;\;\;f|(X\backslash U)\equiv\mathbf{1}.\]
Let
\[C=f^{-1}\big[[0,1/2)\big]\in\mathrm{Coz}(X).\]
Then $C\subseteq U$ and thus $\mathrm{cl}_XC$ has ${\mathscr P}$, as it is  closed in $\mathrm{cl}_XU$. Therefore
\[\mathrm{int}_{\beta X}\mathrm{cl}_{\beta X}C\subseteq\lambda_{\mathscr P} X.\]
But then $x\in\lambda_{\mathscr P} X$, as $x\in f_\beta^{-1}[[0,1/2)]$ and
\[f_\beta^{-1}\big[[0,1/2)\big]\subseteq\mathrm{int}_{\beta X}\mathrm{cl}_{\beta X}C\]
by Lemma \ref{LKG}.
\end{proof}

\begin{remark}
Note that in Lemma \ref{BBV} the converse also holds. That is, $X$ is locally-${\mathscr P}$ whenever $X\subseteq\lambda_{\mathscr P} X$. (For a proof, modify the argument given in Lemma 2.10 of \cite{Ko3}.) However, we will not have any occasion in the sequel to use the converse statement.
\end{remark}

\begin{definition}\label{WWA}
Let $X$ be a completely regular locally-${\mathscr P}$ space, where ${\mathscr P}$ is a closed hereditary topological property. For any $f\in C_b(X)$ denote \[f_\lambda=f_\beta|\lambda_{\mathscr P} X.\]
\end{definition}

Observe that by Lemma \ref{BBV} the function $f_\lambda$ extends $f$.

\begin{lemma}\label{TES}
Let $X$ be a normal locally-${\mathscr P}$ space, where ${\mathscr P}$ is a closed hereditary topological property preserved under finite closed sums. For any $f\in C_b(X)$ the following are equivalent:
\begin{itemize}
\item[\rm(1)] $f\in C_{\mathscr P}(X)$.
\item[\rm(2)] $f_\lambda\in C_c(\lambda_{\mathscr P} X)$.
\end{itemize}
\end{lemma}

\begin{proof}
(1) {\em  implies} (2). Let $T$ be a ${\mathscr P}$-neighborhood of $\mathrm{supp}(f)$ in $X$. Then $\mathrm{supp}(f)\subseteq\mathrm{int}_X T$. Since $X$ is normal, by the Urysohn Lemma, there exists a continuous $g:X\rightarrow[0,1]$ with
\[g|\mathrm{supp}(f)\equiv \mathbf{0}\;\;\;\;\mbox{ and }\;\;\;\;g|(X\backslash\mathrm{int}_X T)\equiv\mathbf{1}.\]
Let
\[C=g^{-1}\big[[0,1/2)\big]\in\mathrm{Coz}(X).\]
Note that
\[\mathrm{cl}_XC\subseteq g^{-1}\big[[0,1/2]\big]\subseteq T,\]
and thus, $\mathrm{cl}_XC$, being closed in $T$, has ${\mathscr P}$. Therefore
\[\mathrm{int}_{\beta X}\mathrm{cl}_{\beta X}C\subseteq\lambda_{\mathscr P} X.\]
But
\[g_\beta^{-1}\big[[0,1/2)\big]\subseteq\mathrm{int}_{\beta X}\mathrm{cl}_{\beta X}C\]
by Lemma \ref{LKG}, and thus
\[\mathrm{cl}_{\beta X}\mathrm{Coz}(f)\subseteq\mathrm{Z}(g_\beta)\subseteq g_\beta^{-1}\big[[0,1/2)\big]\subseteq\lambda_{\mathscr P} X.\]
This implies that
\begin{eqnarray*}
\mathrm{supp}(f_\lambda)&=&\mathrm{cl}_{\lambda_{\mathscr P} X}\mathrm{Coz}(f_\lambda)\\&=&\mathrm{cl}_{\lambda_{\mathscr P} X}\big(X\cap \mathrm{Coz}(f_\lambda)\big)\\&=&\mathrm{cl}_{\lambda_{\mathscr P} X}\mathrm{Coz}(f)=\lambda_{\mathscr P}X\cap\mathrm{cl}_{\beta X}\mathrm{Coz}(f)=\mathrm{cl}_{\beta X}\mathrm{Coz}(f)
\end{eqnarray*}
is compact.

(2) {\em  implies} (1). Let $V$ be an open neighborhood of $\mathrm{supp}(f_\lambda)$ in $\beta X$ with
\[\mathrm{cl}_{\beta X}V\subseteq\lambda_{\mathscr P}X.\]
(Note that $\lambda_{\mathscr P} X$ is open in $\beta X$ by its definition, and $\beta X$, being compact, is normal.) By compactness, we have
\begin{equation}\label{JB}
\mathrm{cl}_{\beta X}V\subseteq\mathrm{int}_{\beta X}\mathrm{cl}_{\beta X}C_1\cup\cdots\cup\mathrm{int}_{\beta X}\mathrm{cl}_{\beta X}C_n
\end{equation}
for some $C_1,\ldots,C_n\in\mathrm{Coz}(X)$ such that each $\mathrm{cl}_XC_1,\ldots,\mathrm{cl}_XC_n$ has ${\mathscr P}$. Intersecting both sides of (\ref{JB}) with $X$, we have
\[\mathrm{cl}_X(X\cap V)\subseteq X\cap\mathrm{cl}_{\beta X}V\subseteq\mathrm{cl}_XC_1\cup\cdots\cup\mathrm{cl}_XC_n=D.\]
Note that $D$ has ${\mathscr P}$, as it is a finite union of its closed ${\mathscr P}$-subspaces. Therefore $\mathrm{cl}_X(X\cap V)$, being closed in $D$, has ${\mathscr P}$. But $\mathrm{cl}_X(X\cap V)$ is a neighborhood of $\mathrm{supp}(f)$ in $X$, as
\[\mathrm{supp}(f)\subseteq X\cap\mathrm{supp}(f_\lambda)\subseteq X\cap V.\]
\end{proof}

A version of the classical Banach--Stone Theorem states that for any locally compact spaces $X$ and $Y$, the rings $C_c(X)$ and $C_c(Y)$ are isomorphic if and only if the spaces $X$ and $Y$ are homeomorphic. (See \cite{ACG} or \cite{A}.) This will be used in the proof of the following theorem.

\begin{theorem}\label{UUS}
Let $X$ be a normal locally-${\mathscr P}$ space where ${\mathscr P}$ is a closed hereditary topological property preserved under finite closed sums. Then $C_{\mathscr P}(X)$ is a normed subalgebra of $C_b(X)$ isometrically isomorphic to $C_c(Y)$ for some unique (up to homeomorphism) locally compact space $Y$, namely $Y=\lambda_{\mathscr P} X$. Furthermore, $C_{\mathscr P}(X)$ is unital if and only if $X$ has ${\mathscr P}$.
\end{theorem}

\begin{proof}
First, we need to show that $C_{\mathscr P}(X)$ is a subalgebra of $C_b(X)$. Observe that since $X$ is locally-${\mathscr P}$ (and non-empty), there exists a ${\mathscr P}$-subspace of $X$ which constitutes a neighborhood of $\emptyset=\mathrm{supp}(\mathbf{0})$ in $X$. Thus $\mathbf{0}\in C_{\mathscr P}(X)$. To show that $C_{\mathscr P}(X)$ is closed under addition, let $f_i\in C_{\mathscr P}(X)$ where $i=1,2$. For each $i=1,2$ let $T_i$ be a ${\mathscr P}$-neighborhood of $\mathrm{supp}(f_i)$ in $X$ and (using normality of $X$) let $U_i$ be an open neighborhood of $\mathrm{supp}(f_i)$ in $X$ with $\mathrm{cl}_XU_i\subseteq\mathrm{int}_X T_i$. Then $\mathrm{cl}_XU_1\cup\mathrm{cl}_XU_2$ has ${\mathscr P}$, as it is the union of two of its closed subspaces $\mathrm{cl}_XU_1$ and $\mathrm{cl}_XU_2$, and $\mathrm{cl}_XU_i$, for each $i=1,2$, being closed in $T_i$, has ${\mathscr P}$. Note that $\mathrm{cl}_XU_1\cup\mathrm{cl}_XU_2$ is a neighborhood of $\mathrm{supp}(f_1+f_2)$ in $X$, as
\[\mathrm{supp}(f_1+f_2)\subseteq\mathrm{supp}(f_1)\cup\mathrm{supp}(f_2)\subseteq U_1\cup U_2.\]
That $C_{\mathscr P}(X)$ is closed under scalar multiplication and multiplication of its elements may be proved analogously.

Let $Y=\lambda_{\mathscr P} X$ and define
\[\psi:C_{\mathscr P}(X)\rightarrow C_c(Y)\]
by
\[\psi(f)=f_\lambda\]
for any $f\in C_{\mathscr P}(X)$. By Lemma \ref{TES} the function $\psi$ is well-defined. It is clear that $\psi$ is a homomorphism and that $\psi$ is injective. (Note that $X\subseteq Y$ by Lemma \ref{BBV}, and that any two scalar-valued continuous functions on $\lambda_{\mathscr P} X$ coincide, provided that they agree on the dense subspace $X$ of $Y$.) To show that $\psi$ is surjective, let $g\in C_c(Y)$. Then $(g|X)_\lambda=g$ and thus  $g|X\in C_{\mathscr P}(X)$ by Lemma \ref{TES}. Now $\psi(g|X)=g$. To show that $\psi$ is an isometry, let $h\in C_{\mathscr P}(X)$. Then
\[|h_\lambda|[\lambda_{\mathscr P} X]=|h_\lambda|[\mathrm{cl}_{\lambda_{\mathscr P} X}X]\subseteq\mathrm{cl}_{\mathbb{R}}\big(|h_\lambda|[X]\big)=\mathrm{cl}_{\mathbb{R}}\big(|h|[X]\big)\subseteq\big[0,\|h\|\big]\]
which yields $\|h_\lambda\|\leq\|h\|$. That $\|h\|\leq\|h_\lambda\|$ is clear, as $h_\lambda$ extends $h$.

Note that $Y$ is locally compact, as it is open in the compact space $\beta X$.

The uniqueness of $Y$ follows from the fact that for any locally compact space $T$ the ring $C_c(T)$ determines the topology of $T$.

For the second part of the theorem, suppose that $X$ has ${\mathscr P}$. Then the function $\mathbf{1}$ is the unit element of $C_{\mathscr P}(X)$. To show the converse, suppose that $C_{\mathscr P}(X)$ has a unit element $u$. Let $x\in X$. Let $U_x$ and $V_x$ be open neighborhoods of $x$ in $X$ such that $\mathrm{cl}_XV_x\subseteq U_x$ and $\mathrm{cl}_XU_x$ has ${\mathscr P}$. Let $f_x:X\rightarrow[0,1]$ be continuous and such that
\[f_x(x)=1\;\;\;\;\mbox{ and }\;\;\;\;f_x|(X\backslash V_x)\equiv\mathbf{0}.\]
Then $\mathrm{cl}_XU_x$ is a neighborhood of $\mathrm{supp}(f_x)$, as $\mathrm{supp}(f_x)\subseteq\mathrm{cl}_XV_x$. Therefore $f_x\in C_{\mathscr P}(X)$. Since
\[u(x)=u(x)f_x(x)=f_x(x)=1\]
we have $u=\mathbf{1}$. Thus $X=\mathrm{supp}(u)$ has ${\mathscr P}$.
\end{proof}

\begin{example}\label{QLL}
The list of topological properties satisfying the assumption of Theorem \ref{UUS} is quite long and include almost all important covering properties (that is, topological properties described in terms of the existence of certain kinds of open subcovers or refinements of a given open cover of a certain type), among them are: compactness, countable compactness (more generally, $[\theta,\kappa]$-compactness), the Lindel\"{o}f property (more generally, the $\mu$-Lindel\"{o}f property), paracompactness, metacompactness, countable paracompactness, subparacompactness, submetacompactness (or $\theta$-refinability), the $\sigma$-para-Lindel\"{o}f property and also $\alpha$-boundedness. (See \cite{Bu} and \cite{Steph} for the definitions. That these topological properties - except for the last one - are closed hereditary and preserved under finite closed sums, follow from Theorems 7.1, 7.3 and 7.4 of \cite{Bu}; for $\alpha$-boundedness, this directly follows from its definition. Recall that a space $X$ is {\em $\alpha$-bounded}, where $\alpha$ is an infinite cardinal, if every subspace of $X$ of cardinality $\leq\alpha$ has compact closure in $X$.)
\end{example}

\begin{remark}
Let ${\mathscr P}$ be a topological property. Then ${\mathscr P}$ is {\em finitely additive}, if any space which is expressible as a finite disjoint union of its closed ${\mathscr P}$-subspaces has ${\mathscr P}$. Also, ${\mathscr P}$ is {\em invariant under perfect mappings} ({\em  inverse invariant under perfect mappings}, respectively) if for every perfect surjective mapping $f:X\rightarrow Y$, the space $Y$ ($X$, respectively) has ${\mathscr P}$, provided that $X$ ($Y$, respectively) has ${\mathscr P}$. If ${\mathscr P}$ is both invariant and inverse invariant under perfect mappings then it is {\em perfect}. (A closed continuous mapping $f:X\rightarrow Y$ is {\em perfect}, if each fiber $f^{-1}(y)$, where $y\in Y$, is a compact subspace of $X$.) Any finitely additive topological property which is invariant under perfect mappings is preserved under finite closed sums. (See Theorem 3.7.22 of \cite{E}.) Also, any topological property which is hereditary with respect to open-closed subspaces and is inverse invariant under perfect mappings, is hereditary with respect to closed subspaces. (See Theorem 3.7.29 of \cite{E}.) Therefore, the assumption that ``${\mathscr P}$ is closed hereditary and preserved under finite closed sums" in Lemma \ref{TES} and Theorem \ref{UUS} may be replaces by ``${\mathscr P}$ is open-closed hereditary, finitely additive and perfect".
\end{remark}

\section{The Banach algebra $C_{\mathscr P}(X)$}

In this section we turn our attention to the case in which $C_{\mathscr P}(X)$ becomes a Banach algebra. It is interesting that in spite of the fact that $C_{\mathscr P}(X)$ is demanded to have a richer structure, it turns out to be better expressible, and at the same time, $\lambda_{\mathscr P} X$ reveals nicer properties. These are all more precisely expressed in the statement of our next result.

Let $X$ be a locally compact non-compact space. It is known that $C_0(X)=C_c(X)$ if and only if every $\sigma$-compact subspace of $X$ is contained in a compact subspace of $X$. (See Problem 7G.2 of \cite{GJ}.) In particular, $C_0(X)=C_c(X)$ implies that $X$ is countably compact (recall that a space $T$ is countably compact if and only if each countably infinite subspace of $T$ has an accumulation point; see Theorem 3.10.3 of \cite{E}) and thus, non-Lindel\"{o}f and non-paracompact, as every countably compact space which is either Lindel\"{o}f or paracompact is necessarily compact. (See Theorems 3.11.1 and 5.1.20 of \cite{E}; observe that Lindel\"{o}f spaces are realcompact and completely regular countably compact spaces are pseudocompact; see Theorems 3.11.12 and 3.10.20 of \cite{E}.) These will be used in the proof of the following.

\begin{theorem}\label{YUS}
Let $X$ be a normal locally-${\mathscr P}$ space where ${\mathscr P}$ is a closed hereditary topological property preserved under countable closed sums. Moreover, suppose that the closure of each ${\mathscr P}$-subspace of $X$ has a ${\mathscr P}$-neighborhood. Then $C_{\mathscr P}(X)$ is a Banach subalgebra of $C_b(X)$ isometrically isomorphic to $C_c(Y)$ for some unique (up to homeomorphism) locally compact space $Y$, namely $Y=\lambda_{\mathscr P} X$. Moreover
\begin{itemize}
  \item $C_{\mathscr P}(X)=\{f\in C_b(X):\mathrm{supp}(f)\mbox{ has }{\mathscr P}\}$.
  \item $C_c(Y)=C_0(Y)$.
  \item $Y$ is countably compact.
  \item $Y$ is neither Lindel\"{o}f nor paracompact, if $X$ is non-${\mathscr P}$.
\end{itemize}
\end{theorem}

\begin{proof}
By Theorem \ref{UUS} we know that $C_{\mathscr P}(X)$ is a normed subalgebra of $C_b(X)$ isometrically isomorphic to $C_c(Y)$ for some unique locally compact space $Y=\lambda_{\mathscr P} X$. To prove that $C_{\mathscr P}(X)$ is a Banach algebra it then suffices to show that $C_c(Y)=C_0(Y)$.

Next, note that if $f\in C_{\mathscr P}(X)$, then $\mathrm{supp}(f)$ has ${\mathscr P}$, as it is closed in a ${\mathscr P}$-neighborhood in $X$. For the converse, note that if $f\in C_b(X)$ is such that $\mathrm{supp}(f)$ has ${\mathscr P}$, than $\mathrm{supp}(f)$ has a ${\mathscr P}$-neighborhood by our assumption.

To show that $C_c(Y)=C_0(Y)$, let $A$ be a $\sigma$-compact subspace of $Y$. Then
\[A=A_1\cup A_2\cup\cdots\]
where each $A_1,A_2,\ldots$ is compact. For each $n=1,2,\ldots$ by compactness of $A_n$ we have
\begin{equation}\label{DSD}
A_n\subseteq\mathrm{int}_{\beta X}\mathrm{cl}_{\beta X}C^n_1\cup\cdots\cup\mathrm{int}_{\beta X}\mathrm{cl}_{\beta X}C^n_{k_n}
\end{equation}
for some $C^n_1,\ldots,C^n_{k_n}\in\mathrm{Coz}(X)$ such that each $\mathrm{cl}_XC^n_1,\ldots,\mathrm{cl}_XC^n_{k_n}$ has ${\mathscr P}$. Note that
\[E=\bigcup_{n=1}^\infty\bigcup_{i=1}^{k_n}\mathrm{cl}_XC^n_i\]
has ${\mathscr P}$, as it is the countable union of its closed ${\mathscr P}$-subspaces. By our assumption, there exists a ${\mathscr P}$-neighborhood $T$ of $\mathrm{cl}_X E$ in $X$. Since $X$ is normal by the Urysohn Lemma there exists a continuous $f:X\rightarrow[0,1]$ with
\[f|\mathrm{cl}_X E\equiv \mathbf{0}\;\;\;\;\mbox{ and }\;\;\;\;f|(X\backslash\mathrm{int}_X T)\equiv \mathbf{1}.\]
Let
\[C=f^{-1}\big[[0,1/2)\big]\in\mathrm{Coz}(X).\]
Note that
\[\mathrm{cl}_XC\subseteq f^{-1}\big[[0,1/2]\big]\subseteq T,\]
and thus, $\mathrm{cl}_XC$, being closed in $T$, has ${\mathscr P}$. Therefore
\[\mathrm{int}_{\beta X}\mathrm{cl}_{\beta X}C\subseteq\lambda_{\mathscr P} X.\]
But
\[f_\beta^{-1}\big[[0,1/2)\big]\subseteq\mathrm{int}_{\beta X}\mathrm{cl}_{\beta X}C\]
by Lemma \ref{LKG}. Thus
\begin{equation}\label{TY}
\mathrm{cl}_{\beta X}C^n_i\subseteq\mathrm{Z}(f_\beta)\subseteq f_\beta^{-1}\big[[0,1/2)\big]\subseteq\lambda_{\mathscr P} X
\end{equation}
for each $n=1,2,\ldots$ and $i=1,\ldots,k_n$. From (\ref{DSD}) and (\ref{TY}), it then follows that $A_n\subseteq\mathrm{Z}(f_\beta)$ for each $n=1,2,\ldots$. Therefore $\mathrm{Z}(f_\beta)$ is a compact subspaces of $\lambda_{\mathscr P} X$ containing $A$.

To conclude the proof, note that $Y$ is countably compact, as $C_c(Y)=C_0(Y)$. If $Y$ is in addition either Lindel\"{o}f or paracompact, then it is compact. Compactness of $Y$ now implies that
\[\lambda_{\mathscr P} X=\mathrm{int}_{\beta X}\mathrm{cl}_{\beta X}C_1\cup\cdots\cup\mathrm{int}_{\beta X}\mathrm{cl}_{\beta X}C_n\]
for some $C_1,\ldots,C_n\in\mathrm{Coz}(X)$ such that each $\mathrm{cl}_XC_1,\ldots,\mathrm{cl}_XC_n$ has ${\mathscr P}$. Since $X$ is locally-${\mathscr P}$, we have $X\subseteq\lambda_{\mathscr P} X$ by Lemma \ref{BBV}, from which it then follows that
\[X=\mathrm{cl}_XC_1\cup\cdots\cup\mathrm{cl}_XC_n,\]
being the finite union of its closed ${\mathscr P}$-subspaces, has ${\mathscr P}$.
\end{proof}

\begin{remark}\label{JIJ}
Suppose that the underlying field of scalars is $\mathbb{C}$. In Theorem \ref{YUS} we have proved that $C_{\mathscr P}(X)$ is a Banach algebra isometrically isomorphic to $C_0(Y)$ for some locally compact space $Y$ ($=\lambda_{\mathscr P} X$). On the other hand, by the commutative Gelfand--Naimark Theorem, we know that $C_{\mathscr P}(X)$ is isometrically isomorphic to $C_0(Y')$, with the locally compact space $Y'$ being the spectrum of $C_{\mathscr P}(X)$. Thus $C_0(Y)$ and $C_0(Y')$ are isometrically isomorphic, which implies that the spaces $Y$ and $Y'$ are homeomorphic. In particular, this shows that $\lambda_{\mathscr P} X$ coincides with the spectrum of $C_{\mathscr P}(X)$. (Recall that, by a version of the classical Banach--Stone Theorem, for any locally compact spaces $X$ and $Y$, the Banach algebras $C_0(X)$ and $C_0(Y)$ are isometrically isomorphic if and only if the spaces $X$ and $Y$ are homeomorphic; see Theorem 7.1 of \cite{Be}.)
\end{remark}

\begin{remark}\label{GGJ}
Note that in Theorem \ref{YUS} the space $Y$ is non-${\mathscr P}$ for any topological property ${\mathscr P}$ such that
\[\mbox{${\mathscr P}$ $+$ countable compactness $\rightarrow$ compactness.}\]
The list of such topological properties is quite long; it includes (in addition to the Lindel\"{o}f property and paracompactness themselves): realcompactness, metacompactness, subparacompactness, submetacompactness (or $\theta$-refinability), the meta-Lindel\"{o}f property, the submeta-Lindel\"{o}f property (or $\delta\theta$-refinability), weak submetacompactness (or weak $\theta$-refinability) and the weak submeta-Lindel\"{o}f property (or weak $\delta\theta$-refinability) among others. (See Parts 6.1 and 6.2 of \cite{Va}.)
\end{remark}

\section{The case when ${\mathscr P}$ is the Lindel\"{o}f property}

In this section we confine ourselves to the case when ${\mathscr P}$ is the Lindel\"{o}f property. This consideration leads to some improvements in Theorem \ref{YUS}.
Part of the results of this section (Lemmas \ref{PDS}, \ref{OPS} and \ref{RTS} and part of Theorem \ref{RTF}) are slight modifications of certain results from \cite{Ko6}. The proofs are given here for reader's convenience and completeness of results.

We begin with the following observation.

\begin{proposition}\label{GGS}
Let ${\mathscr P}$ be the Lindel\"{o}f property. Let $X$ be a normal locally-${\mathscr P}$ space. Then
\[C_{\mathscr P}(X)=\big\{f\in C_b(X):\mathrm{supp}(f)\mbox{ has }{\mathscr P}\big\}.\]
\end{proposition}

\begin{proof}
Observe that if $f\in C_{\mathscr P}(X)$ then $\mathrm{supp}(f)$ has ${\mathscr P}$, as it is closed in a ${\mathscr P}$-neighborhood in $X$.

Next, suppose that $f\in C_b(X)$ and that $\mathrm{supp}(f)$ has ${\mathscr P}$. For each $x\in\mathrm{supp}(f)$, let $U_x$ be an open neighborhood of $x$ in $X$ such that the closure $\mathrm{cl}_X U_x$ has ${\mathscr P}$. Since
\[\big\{U_x:x\in\mathrm{supp}(f)\big\}\]
is an open cover of $\mathrm{supp}(f)$, there exist some $x_1,x_2,\ldots\in\mathrm{supp}(f)$ such that
\[\mathrm{supp}(f)\subseteq U_{x_1}\cup U_{x_2}\cup\cdots=W.\]
By normality of $X$, there exists an open subspace $V$ of $X$ with
\[\mathrm{supp}(f)\subseteq V\subseteq\mathrm{cl}_X V\subseteq W.\]
Now $\mathrm{cl}_X V$ is contained in \[H=\mathrm{cl}_XU_{x_1}\cup\mathrm{cl}_XU_{x_2}\cup\cdots\]
as a closed subspace. Since $H$ has ${\mathscr P}$, it follows that $\mathrm{cl}_X V$ has ${\mathscr P}$. That is, $\mathrm{cl}_X V$ is a ${\mathscr P}$-neighborhood of $\mathrm{supp}(f)$ in $X$. Therefore $f\in C_{\mathscr P}(X)$.
\end{proof}

\begin{lemma}\label{PDS}
Let ${\mathscr P}$ be the Lindel\"{o}f property. Let $X$ be a completely regular space representable as a disjoint union
\[X=\bigcup_{i\in I}X_i,\]
such that $X_i$'s are open-closed ${\mathscr P}$-subspaces of $X$. Then
\[\lambda_{\mathscr P} X=\bigcup\Big\{\mathrm{cl}_{\beta X}\Big(\bigcup_{i\in J}X_i\Big):J\subseteq I\mbox{ is countable}\Big\}.\]
\end{lemma}

\begin{proof}
Let
\[Y=\bigcup\Big\{\mathrm{cl}_{\beta X}\Big(\bigcup_{i\in J}X_i\Big):J\subseteq I\mbox{ is countable}\Big\}.\]

To show that $\lambda_{\mathscr P} X\subseteq Y$, let $C\in\mathrm{Coz}(X)$ has Lindel\"{o}f closure $\mathrm{cl}_X C$. Then
\[\mathrm{cl}_X C\subseteq\bigcup_{i\in J}X_i\]
for some countable $J\subseteq I$. Thus
\[\mathrm{cl}_{\beta X}C\subseteq\mathrm{cl}_{\beta X}\Big(\bigcup_{i\in J}X_i\Big).\]

We next show that $Y\subseteq\lambda_{\mathscr P} X$. Let $J\subseteq I$ be countable. Then
\[D=\bigcup_{i\in J}X_i\]
is a cozero-set of $X$, as it is open-closed in $X$, and it is Lindel\"{o}f. Since $D$ is open-closed in $X$, the closure $\mathrm{cl}_{\beta X}D$ in $\beta X$ is open-closed in $\beta X$. Therefore
\[\mathrm{cl}_{\beta X}D=\mathrm{int}_{\beta X}\mathrm{cl}_{\beta X}D\subseteq\lambda_{\mathscr P} X.\]
\end{proof}

Let ${\mathscr P}$ be the Lindel\"{o}f property. Let $D$ be an uncountable discrete space. Let $E$ be the subspace of $\beta D\backslash D$ consisting of elements in the closure (in $\beta D$) of countable subspaces of $D$. Then
\[E=\lambda_{\mathscr P} D\backslash D.\]
(Observe that cozero-sets in $D$ whose closure in $D$ has ${\mathscr P}$ are exactly countable subspace of $D$, and that each subspace of $D$, being open-closed in $D$, has open closure in $\beta D$.) In \cite{W}, the author proves the existence of a continuous (2-valued) function $f:E\rightarrow[0,1]$ which is not continuously extendible over $\beta D\backslash D$. This, in particular, proves that $\lambda_{\mathscr P} D$ is not normal. (To see this, suppose, in the contrary, that $\lambda_{\mathscr P} D$ is normal. Note that $E$ is closed in $\lambda_{\mathscr P} D$, as $D$, being locally compact, is open in $\beta D$. By the Tietze--Urysohn Extension Theorem, $f$ is extendible to a continuous bounded function over $\lambda_{\mathscr P} D$, and thus over $\beta(\lambda_{\mathscr P} D)$. Note that $\beta(\lambda_{\mathscr P} D)=\beta D$, as $D\subseteq\lambda_{\mathscr P} D$ by Lemma \ref{BBV}. But this is not possible.) This fact will be used in the following to show that in general $\lambda_{\mathscr P} X$ need not be normal. This, in particular, provides us with an example of a locally compact countably compact non-normal space $Y$ such that $C_c(Y)=C_0(Y)$.

Observe that if $X$ is a space and $D\subseteq X$, then
\[U\cap\mathrm{cl}_XD=\mathrm{cl}_X(U\cap D)\]
for every open-closed subspace $U$ of $X$. This simple observation will be used below.

\begin{lemma}\label{OPS}
Let ${\mathscr P}$ be the Lindel\"{o}f property. Let $X$ be a completely regular non-${\mathscr P}$-space representable as a disjoint union
\[X=\bigcup_{i\in I}X_i,\]
such that $X_i$'s are open-closed ${\mathscr P}$-subspaces of $X$. Then $\lambda_{\mathscr P} X$ in non-normal.
\end{lemma}

\begin{proof}
Let $x_i\in X_i$ for each $i\in I$. Then
\[D=\{x_i:i\in I\}\]
is a closed discrete subspace of $X$, and since $X$ is non-${\mathscr P}$, is uncountable. Suppose in the contrary that $\lambda_{\mathscr P} X$ is normal. Then  \[\lambda_{\mathscr P} X\cap\mathrm{cl}_{\beta X}D\]
is normal, as it is closed in $\lambda_{\mathscr P} X$. By Lemma \ref{PDS} we have
\[\lambda_{\mathscr P} X\cap\mathrm{cl}_{\beta X}D=\bigcup\Big\{\mathrm{cl}_{\beta X}\Big(\bigcup_{i\in J}X_i\Big)\cap\mathrm{cl}_{\beta X}D:J\subseteq I\mbox{ is countable}\Big\}.\]
Let $J\subseteq I$ be countable. Since
\[\mathrm{cl}_{\beta X}\Big(\bigcup_{i\in J}X_i\Big)\]
is open-closed in $\beta X$ (as $\bigcup_{i\in J}X_i$ is open-closed in $X$) we have
\begin{eqnarray*}
\mathrm{cl}_{\beta X}\Big(\bigcup_{i\in J}X_i\Big)\cap\mathrm{cl}_{\beta X}D&=&\mathrm{cl}_{\beta X}\Big(\mathrm{cl}_{\beta X}\Big(\bigcup_{i\in J}X_i\Big)\cap D\Big)\\&=&\mathrm{cl}_{\beta X}\Big(\bigcup_{i\in J}X_i\cap D\Big)=\mathrm{cl}_{\beta X}\big(\{x_i:i\in J\}\big).
\end{eqnarray*}
But $\mathrm{cl}_{\beta X}D=\beta D$, as $D$ is closed in (the normal space) $X$. Therefore
\[\mathrm{cl}_{\beta X}\big(\{x_i:i\in J\}\big)=\mathrm{cl}_{\beta X}\big(\{x_i:i\in J\}\big)\cap\mathrm{cl}_{\beta X}D=\mathrm{cl}_{\beta D}\big(\{x_i:i\in J\}\big).\]
Thus
\[\lambda_{\mathscr P} X\cap\mathrm{cl}_{\beta X}D=\lambda_{\mathscr P} D,\]
contradicting the fact that $\lambda_{\mathscr P} D$ in not normal.
\end{proof}

The following corollary of Theorem \ref{YUS}, together with Theorem \ref{RTF}, constitute the main result of this section.

\begin{theorem}\label{RDF}
Let ${\mathscr P}$ be the Lindel\"{o}f property. Let $X$ be a paracompact locally-${\mathscr P}$ space. Then $C_{\mathscr P}(X)$ is a Banach subalgebra of $C_b(X)$ isometrically isomorphic to $C_c(Y)$ for some unique (up to homeomorphism) locally compact space $Y$, namely $Y=\lambda_{\mathscr P} X$. Moreover
\begin{itemize}
\item[\rm(1)] $C_{\mathscr P}(X)=\{f\in C_b(X):\mathrm{supp}(f)\mbox{ has }{\mathscr P}\}$.
\item[\rm(2)] $C_0(X)\subseteq C_{\mathscr P}(X)$, with proper inclusion if $X$ is non-${\mathscr P}$.
\item[\rm(3)] $C_c(Y)=C_0(Y)$.
\item[\rm(4)] $Y$ is countably compact.
\item[\rm(5)]  $Y$ is neither Lindel\"{o}f nor paracompact, if $X$ is non-${\mathscr P}$.
\end{itemize}
If $X$ is moreover locally compact then in addition we have
\begin{itemize}
\item[\rm(6)] $C_{\mathscr P}(X)=\{f\in C_b(X):\mathrm{supp}(f)\mbox{ is $\sigma$-compact}\}$.
\item[\rm(7)] $Y$ is non-normal, if $X$ is non-${\mathscr P}$.
\end{itemize}
\end{theorem}

\begin{proof}
Conditions (1), (3), (4) and (5) follow from Theorem \ref{YUS}; we only need to show that the closure in $X$ of each Lindel\"{o}f subspace of $X$ has a Lindel\"{o}f neighborhood in $X$. (Note that the Lindel\"{o}f property is closed hereditary and is preserved under countable closed sums.)

Let $A$ be a Lindel\"{o}f subspace of $X$. Since paracompactness is closed hereditary (see Theorem 5.1.28 of \cite{E}), $\mathrm{cl}_X A$, being closed in $X$, is paracompact. Since any paracompact space having a dense Lindel\"{o}f subspace is itself Lindel\"{o}f (see Theorem 5.1.25 of \cite{E}), $\mathrm{cl}_X A$ is Lindel\"{o}f. For each $x\in\mathrm{cl}_X A$, let $U_x$ be an open neighborhood of $x$ in $X$ with Lindel\"{o}f closure $\mathrm{cl}_X U_x$. Then
\[\mathrm{cl}_X A\subseteq U_{x_1}\cup U_{x_2}\cup\cdots=U\]
for some $x_1,x_2,\ldots\in\mathrm{cl}_X A$. Since $X$ is normal (as it is paracompact) there exists an open neighborhood $V$ of $\mathrm{cl}_X A$ in $X$ such that $\mathrm{cl}_X V\subseteq U$. Observe that $\mathrm{cl}_X V$ is Lindel\"{o}f, as it is a closed subspace of the Lindel\"{o}f space
\[\mathrm{cl}_X U_{x_1}\cup\mathrm{cl}_X U_{x_2}\cup\cdots.\]

To show (2), let $f\in C_0(X)$. Then $|f|^{-1}([1/n,\infty))$ is compact for each $n=1,2,\ldots$ and therefore
\[\mathrm{Coz}(f)=\bigcup_{n=1}^\infty |f|^{-1}\big([1/n,\infty)\big)\]
is $\sigma$-compact and thus Lindel\"{o}f. Note that any paracompact space with a dense Lindel\"{o}f subspace is Lindel\"{o}f. (See Theorem 5.1.25 of \cite{E}.) Since paracompactness is closed hereditary (see Theorem 5.1.28 of \cite{E}) $\mathrm{supp}(f)$ is paracompact, as it is closed in $X$, and thus it is Lindel\"{o}f, as it contains $\mathrm{Coz}(f)$ as a dense subspace. Therefore $f\in C_{\mathscr P}(X)$. Now suppose that $X$ is non-Lindel\"{o}f. Assume the representation of $X$ given in Part 1.3. We may further assume that $X_i$'s are non-compact. (Otherwise, group together any countable number of $X_i$'s.) Then, for the function $f$ which is defined to be identical to $1$ on $X_i$ and vanishing elsewhere, we have $f\in C_{\mathscr P}(X)$, while trivially $f\notin C_0(X)$.

In the remainder of the proof assume that $X$ is moreover locally compact.

Note that (7) follows from Lemma \ref{OPS} (using the representation of $X$ given in Part 1.3).

To show (6), let $f\in C_b(X)$. If $f\in C_{\mathscr P}(X)$, then since $X$ is normal, $\mathrm{supp}(f)$ has a closed Lindel\"{o}f neighborhood in $X$. But then $\mathrm{supp}(f)$ has a closed $\sigma$-compact neighborhood in $X$, as any closed neighborhood of $\mathrm{supp}(f)$ in $X$ (being closed in the locally compact space $X$) is locally compact, and in the realm of locally compact spaces, the two notions of $\sigma$-compactness and being Lindel\"{o}f coincide. (See Problem 3.8.C of \cite{E}.) Thus $\mathrm{supp}(f)$ is $\sigma$-compact. For the converse, note that if $\mathrm{supp}(f)$ is $\sigma$-compact, than it is Lindel\"{o}f, and therefore $f\in C_{\mathscr P}(X)$ by (1).
\end{proof}

Let $I$ be an infinite set. A theorem of Tarski guarantees the existence of a collection $\mathscr{I}$ of cardinality $|I|^{\aleph_0}$ consisting of countable infinite subsets of $I$, such that the intersection of any two distinct elements of $\mathscr{I}$ is finite (see Theorem 2.1 of \cite{Ho}); this will be used in the following.

Note that the collection of all subsets of cardinality at most $\mathfrak{m}$ in a set of cardinality $\mathfrak{n}\geq\mathfrak{m}$ has cardinality at most $\mathfrak{n}^\mathfrak{m}$.

\begin{lemma}\label{RTS}
Let ${\mathscr P}$ be the Lindel\"{o}f property. Let $X$ be a locally-${\mathscr P}$ non-${\mathscr P}$ metrizable space. Then
\[\dim C_{\mathscr P}(X)=\ell(X)^{\aleph_0}.\]
\end{lemma}

\begin{proof}
Assume the representation of $X$ given in Part 1.4. Note that $I$ is infinite, as $X$ is non-Lindel\"{o}f, and $\ell(X)=|I|$.

Let $\mathscr{I}$ be a collection of cardinality $|I|^{\aleph_0}$ consisting of countable infinite subsets of $I$, such that the intersection of any two distinct elements of $\mathscr{I}$ is finite. To simplify the notation denote
\[H_J=\bigcup_{i\in J}X_i\]
for each $J\subseteq I$. Define
\[f_J=\chi_{H_J}\]
for any $J\in\mathscr{I}$. Then, no element of
\[\mathscr{F}=\{f_J:J\in\mathscr{I}\}\]
is a linear combination of other elements (since each element of $\mathscr{I}$ is infinite and each pair of distinct elements of $\mathscr{I}$ has finite intersection). Observe that $\mathscr{F}$ is of cardinality $|\mathscr{I}|$. This shows that
\[\dim C_{\mathscr P}(X)\geq|\mathscr{I}|=|I|^{\aleph_0}=\ell(X)^{\aleph_0}.\]

Note that if $f\in C_{\mathscr P}(X)$, then since $\mathrm{supp}(f)$ is Lindel\"{o}f, we have
\[\mathrm{supp}(f)\subseteq H_J\]
where $J\subseteq I$ is countable, therefore, one may assume that $f\in C_b(H_J)$. Conversely, if $J\subseteq I$ is countable, then each element of $C_b(H_J)$ can be extended trivially to an element of $C_{\mathscr P}(X)$ (by defining it to be identically $0$ elsewhere). Thus $C_{\mathscr P}(X)$ may be viewed as the union of all $C_b(H_J)$, where $J$ runs over all countable subsets of $I$. Note that if $J\subseteq I$ is countable, then $H_J$ is separable (note that in metrizable spaces separability coincides with being Lindel\"{o}f); thus any element of $C_b(H_J)$ is determined by its value on a countable set. This implies that for each countable $J\subseteq I$, the set $C_b(H_J)$ is of cardinality at most $\mathfrak{c}^{\aleph_0}=2^{\aleph_0}$. Observe that there exist at most $|I|^{\aleph_0}$ countable $J\subseteq I$. Now
\begin{eqnarray*}
\dim C_{\mathscr P}(X)\leq\big|C_{\mathscr P}(X)\big|&\leq&\Big|\bigcup\big\{C_b(H_J):J\subseteq I\mbox{ is countable}\big\}\Big|\\&\leq& 2^{\aleph_0}\cdot|I|^{\aleph_0}=|I|^{\aleph_0}=\ell(X)^{\aleph_0},
\end{eqnarray*}
which together with the first part proves the lemma.
\end{proof}

The following is a counterpart of Theorem \ref{RDF}.

\begin{theorem}\label{RTF}
Let ${\mathscr P}$ be the Lindel\"{o}f property. Let $X$ be a metrizable locally-${\mathscr P}$ space. Then $C_{\mathscr P}(X)$ is a Banach subalgebra of $C_b(X)$ isometrically isomorphic to $C_c(Y)$ for some unique (up to homeomorphism) locally compact space $Y$, namely $Y=\lambda_{\mathscr P} X$. Moreover
\begin{itemize}
\item[\rm(1)] $C_{\mathscr P}(X)=\{f\in C_b(X):\mathrm{supp}(f)\mbox{ has }{\mathscr P}\}$.
\item[\rm(2)] $C_0(X)\subseteq C_{\mathscr P}(X)$, with proper inclusion if $X$ is non-${\mathscr P}$.
\item[\rm(3)] $C_c(Y)=C_0(Y)$.
\item[\rm(4)] $Y$ is countably compact.
\item[\rm(5)] $Y$ is non-normal, if $X$ is non-${\mathscr P}$.
\item[\rm(6)] $\dim C_{\mathscr P}(X)=\ell(X)^{\aleph_0}$.
\end{itemize}
\end{theorem}

\begin{proof}
The theorem follows from Lemmas \ref{OPS} and \ref{RTS} and Theorem \ref{RDF}. Observe that metrizable spaces are paracompact.
\end{proof}

\begin{remark}
Note that in metrizable spaces the notions of second countability and separability coincide with being Lindel\"{o}f. Thus, in Theorem \ref{RTF} we have
\begin{eqnarray*}
C_{\mathscr P}(X)&=&\big\{f\in C_b(X):\mathrm{supp}(f)\mbox{ is separable}\big\}\\&=&\big\{f\in C_b(X):\mathrm{supp}(f)\mbox{ is second countable}\big\}.
\end{eqnarray*}
\end{remark}

\begin{remark}\label{HHJ}
Theorems \ref{RDF} and \ref{RTF} highlight how differently $C_{\mathscr P}(X)$ may behave by varying the topological property ${\mathscr P}$ for a locally-${\mathscr P}$ space $X$. If ${\mathscr P}$ is compactness then of course $C_{\mathscr P}(X)=C_c(X)$. Thus, in this case, if $C_{\mathscr P}(X)$ is a Banach algebra, then it is closed in $C_0(X)$, and since $C_c(X)$ is dense in $C_0(X)$, it follows that $C_0(X)=C_c(X)$. It is known that for any locally compact space $Y$ we have $C_0(Y)=C_c(Y)$ if and only if every $\sigma$-compact subspace of $Y$ is contained in a compact subspace of $Y$. (See Problem 7G.2 of \cite{GJ}.) Thus, in particular, if $Y$ is a locally compact space, then $C_0(Y)=C_c(Y)$ implies that $Y$ is countably compact. From this it follows that $X$ is countably compact, and thus compact, if $X$ is metrizable. In other words, if ${\mathscr P}$ is compactness and $X$ is a locally-${\mathscr P}$ metrizable space, then $C_{\mathscr P}(X)$ being a Banach algebra, implies that $X$ has ${\mathscr P}$. However, if we let ${\mathscr P}$ to be the Lindel\"{o}f property, then for any locally-${\mathscr P}$ metrizable (or even paracompact) space $X$, it follows that $C_{\mathscr P}(X)$ is a Banach algebra, without $X$ necessarily having ${\mathscr P}$.
\end{remark}

\begin{remark}\label{KHJ}
A regular space $X$ is {\em linearly Lindel\"{o}f} if every linearly ordered (by $\subseteq$) open cover of $X$ has a countable subcover; equivalently, if every uncountable subspace of $X$ has a complete accumulation point in $X$; see\cite{G}. (A point $x\in X$ is a {\em complete accumulation point} of a subspace $A$ of $X$, if
\[|U\cap A|=|A|\]
for every neighborhood $U$ of $x$ in $X$.) Obviously, if $X$ is Lindel\"{o}f, then it is linearly Lindel\"{o}f. The converse holds if $X$ is either locally compact paracompact or locally Lindel\"{o}f metrizable. To show this, note that
\[X=\bigcup_{i\in I}X_i,\]
where $X_i$'s are pairwise disjoint non-empty Lindel\"{o}f open-closed subspaces of $X$. (See Parts 1.3 and 1.4.) Now, if $X$ is non-Lindel\"{o}f then $I$ is uncountable, and thus there exists an infinite subspace $A$ of $X$ (choose some $x_i\in X_i$ for each $i\in I$ and let $A=\{x_i:i\in I\}$) without even an accumulation point. That is, $X$ is not linearly Lindel\"{o}f. Since paracompactness, local compactness and being locally Lindel\"{o}f are all closed hereditary, in Theorem \ref{RDF}(4) and Theorem \ref{RTF} (using condition (1) of Theorems \ref{RDF} and \ref{RTF}, respectively) we further have
\[C_{\mathscr P}(X)=\big\{f\in C_b(X):\mathrm{supp}(f)\mbox{ is linearly Lindel\"{o}f}\big\}.\]
\end{remark}

\begin{remark}\label{GHJ}
The {\em density} of a space $X$, denoted by $d(X)$, is defined by
\[d(X)=\min\big\{|D|:D\mbox{ is dense in }X\big\}+\aleph_0.\]
In particular, a space $X$ is separable if and only if $d(X)=\aleph_0$. Note that, if $X$ is a locally Lindel\"{o}f metrizable space, then $d(X)=\ell(X)$. (To see this, assume the representation of $X$ given in Part 1.4 and observe that $d(X)=|I|=\ell(X)$ if $I$ is infinite and $d(X)=\aleph_0=\ell(X)$ otherwise.) Thus in Theorem \ref{RTF}(5) we may replace $\ell(X)$ by $d(X)$.
\end{remark}

In our next result in this section we fit certain type of ideals between $C_0(X)$ and $C_b(X)$.

Let $\mu$ be an infinite cardinal. A regular space $X$ is called {\em $\mu$-Lindel\"{o}f} if every open cover of $X$ has a subcover of cardinality $\leq\mu$. Note that the $\mu$-Lindel\"{o}f property turns weaker as $\mu$ increases. Since the $\aleph_0$-Lindel\"{o}f property coincides with the Lindel\"{o}f property it then follows that every Lindel\"{o}f space is $\mu$-Lindel\"{o}f.

\begin{theorem}\label{GGFD}
Let $X$ be a non-Lindel\"{o}f space which is either locally compact paracompact or locally Lindel\"{o}f metrizable. Then there exists a chain
\[C_0(X)\subsetneqq H_0\subsetneqq H_1\subsetneqq\cdots\subsetneqq H_\lambda=C_b(X)\]
of Banach subalgebras of $C_b(X)$ such that $H_\mu$, for each $\mu\leq\lambda$, is an ideal of $C_b(X)$ isometrically isomorphic to
\[C_0(Y_\mu)=C_c(Y_\mu)\]
for some locally compact space $Y_\mu$. Furthermore, $\aleph_\lambda$ equals the Lindel\"{o}f number $\ell(X)$ of $X$.
\end{theorem}

\begin{proof}
For each ordinal $\mu$, let ${\mathscr P}_\mu$ denote the $\aleph_\mu$-Lindel\"{o}f property, and let
\[H_\mu=C_{{\mathscr P}_\mu}(X).\]

Let $\mu$ be an ordinal. Note that $X$ is normal, and it is locally $\aleph_\mu$-Lindel\"{o}f, as it is locally Lindel\"{o}f. Also, the $\aleph_\mu$-Lindel\"{o}f property, by its definition, is closed hereditary and preserved under countable closed sums. Thus, to use Theorem \ref{YUS}, we only need to show that the closure in $X$ of each $\aleph_\mu$-Lindel\"{o}f subspace of $X$ has a $\aleph_\mu$-Lindel\"{o}f neighborhood in $X$. Assume the representation of $X$ given in Parts 1.3 and 1.4 and note that $\ell(X)=|I|$. Suppose that $A$ is a $\aleph_\mu$-Lindel\"{o}f subspace of $X$. Since
\[\{X_i:i\in I\}\]
is an open cover of $A$, there exists some $J\subseteq I$ with $|J|\leq\aleph_\mu$ such that
\[A\subseteq\bigcup_{i\in J}X_i.\]
If we let
\[U=\bigcup_{i\in J}X_i,\]
then $U$ is a neighborhood of $\mathrm{cl}_XA$ in $X$, and it is $\aleph_\mu$-Lindel\"{o}f, as it is the union of $\aleph_\mu$ number of its Lindel\"{o}f subspaces. By Theorem \ref{YUS} we then know that $H_\mu$ is a Banach subalgebra of $C_b(X)$ isometrically isomorphic to
\[C_0(Y_\mu)=C_c(Y_\mu)\]
for some locally compact space $Y_\mu$; furthermore, we have
\begin{equation}\label{DTT}
H_\mu=\big\{h\in C_b(X):\mathrm{supp}(h)\mbox{ is $\aleph_\mu$-Lindel\"{o}f}\big\}.
\end{equation}
That $H_\mu$ is an ideal of $C_b(X)$ follows easily, as if $h\in H_\mu$, then $\mathrm{supp}(fh)$, for any $f\in C_b(X)$, is $\aleph_\mu$-Lindel\"{o}f, as it is closed in $\mathrm{supp}(h)$ and the latter is $\aleph_\mu$-Lindel\"{o}f, and thus $fh\in H_\mu$.

Note that if $\mu\leq\kappa$ then $H_\mu\subseteq H_\kappa$ by (\ref{DTT}). Let $\lambda$ be such that $\aleph_\lambda=\ell(X)$. Note that $X$ is $\aleph_\lambda$-Lindel\"{o}f (as it is the union of $\aleph_\lambda$ number of its Lindel\"{o}f subspaces). This implies that $H_\lambda=C_b(X)$, as if $f\in C_b(X)$, then $\mathrm{supp}(f)$ is $\aleph_\lambda$-Lindel\"{o}f, as it is closed in $X$. We now show that the inclusions in the chain are all proper. First, note that by Theorems \ref{RDF} and \ref{RTF}, we have $C_0(X)\subsetneqq H_0$. Now, let $\mu<\kappa\leq\lambda$. Let $J\subseteq I$ be of cardinality $\aleph_\kappa$. Then, for the function $f$ which is identical to $1$ on $\bigcup_{i\in J}X_i$ and vanishing elsewhere, we have $f\in H_\kappa$, while $f\notin H_\mu$.
\end{proof}

\begin{remark}\label{HFJ}
Uncountable limit regular cardinals are referred to as {\em weakly inaccessible cardinals}. Weakly inaccessible cardinals cannot be proved to exist within \textsf{ZFC}, though their existence is not known to be inconsistent with \textsf{ZFC}. The existence of weakly inaccessible cardinals is sometimes taken as an additional axiom. Note that weakly inaccessible cardinals are necessarily aleph function's fixed points, that is, if $\lambda$ is a weakly inaccessible cardinal, then $\aleph_\lambda=\lambda$. It is worth noting that in Theorem \ref{GGFD}, if the Lindel\"{o}f number $\ell(X)$ of $X$ is weakly inaccessible, then the chain is of length $\ell(X)$.
\end{remark}

\section{The case when ${\mathscr P}$ is countable compactness}

In this section we determine $\lambda_{\mathscr P}X$ in the case when $X$ is normal and ${\mathscr P}$ is countable compactness. This may also be deduced from Lemma 2.17 of \cite{Ko4} (see also \cite{Ko7}), observing that normality is hereditary with respect to closed subspaces and in the realm of normal spaces countable compactness and pseudocompactness coincide; see Theorems 3.10.20 and 3.10.21 of \cite{E}. (Recall that a completely regular space $X$ is {\em pseudocompact}, if every continuous $f:X\rightarrow\mathbb{R}$ is bounded.) We include the proof here for completeness of results and reader's convenience.

The following result is due to  A.W. Hager and D.G. Johnson in \cite{HJ}; a direct proof may be found in \cite{C}. (See also Theorem 11.24 of \cite{We}.)

\begin{lemma}[Hager--Johnson \cite{HJ}]\label{A}
Let $U$ be an open subspace of a completely regular space $X$. If $\mathrm{cl}_{\upsilon X} U$ is compact then $\mathrm{cl}_X U$ is pseudocompact.
\end{lemma}

Observe, in the proof of the following, that realcompactness is closed hereditary, a space having a pseudocompact dense subspace is pseudocompact, and that realcompact pseudocompact spaces are compact; see Theorems 3.11.1 and 3.11.4 of \cite{E}.

\begin{lemma}\label{HGA}
Let $U$ be an open subspace of a completely regular space $X$. Then $\mathrm{cl}_{\beta X} U\subseteq\upsilon X$ if and only if $\mathrm{cl}_X U$ is pseudocompact.
\end{lemma}

\begin{proof}
The first half of the lemma follows from Lemma \ref{A}. For the second half, note that if $A=\mathrm{cl}_X U$ is pseudocompact then so is its closure $\mathrm{cl}_{\upsilon X}A$. But  $\mathrm{cl}_{\upsilon X} A$, being closed in $\upsilon X$, is also realcompact, and thus compact. Therefore $\mathrm{cl}_{\beta X} A\subseteq\mathrm{cl}_{\upsilon X} A$.
\end{proof}

\begin{theorem}\label{PTF}
Let ${\mathscr P}$ be countable compactness. Let $X$ be a normal space. Then
\[\lambda_{\mathscr P}X=\mathrm{int}_{\beta X}\upsilon X.\]
\end{theorem}

\begin{proof}
If $C\in\mathrm{Coz}(X)$ has countably compact (and thus pseudocompact) closure in $X$, then $\mathrm{cl}_{\beta X} C\subseteq\upsilon X$, by Lemma \ref{HGA}, and then \[\mathrm{int}_{\beta X}\mathrm{cl}_{\beta X} C\subseteq\mathrm{int}_{\beta X}\upsilon X.\]

For the reverse inclusion, let $t\in\mathrm{int}_{\beta X}\upsilon X$. Let $f:\beta X\rightarrow[0,1]$ be continuous with
\[f(t)=0\;\;\;\;\mbox{ and }\;\;\;\;f|(\beta X\backslash\mathrm{int}_{\beta X}\upsilon X)\equiv\mathbf{1}.\]
Then
\[C=X\cap f^{-1}\big[[0,1/2)\big]\in\mathrm{Coz}(X)\]
and $t\in\mathrm{int}_{\beta X}\mathrm{cl}_{\beta X} C$ by Lemma \ref{LKG}. (Note that $(f|X)_\beta=f$, as they coincide on the dense subspace $X$ of $\beta X$.) Also, $\mathrm{cl}_X C$ is pseudocompact by Lemma \ref{HGA}, as
\[\mathrm{cl}_{\beta X}C\subseteq f^{-1}\big[[0,1/2]\big]\subseteq\upsilon X,\]
and thus it is countably compact, since (being closed in $X$) it is normal.
\end{proof}

\section*{Acknowledgement}

The author wishes to thank the referee for reading the manuscript and his/her comments and suggestions.

\bibliographystyle{amsplain}

\begin{thebibliography}{10}

\bibitem{AG} S.K. Acharyya and S.K. Ghosh, Functions in $C(X)$ with support lying on a class of subsets of $X$. \textit{Topology Proc.} \textbf{35} (2010), 127-–148.

\bibitem{ACG} S.K. Acharyya, K.C. Chattopadhyay and P.P. Ghosh, The rings $C_K(X)$ and $C_\infty(X)$ - some remarks. \textit{Kyungpook Math. J.} \textbf{43} (2003), 363--369.

\bibitem{AN} S. Afrooz and M. Namdari, $C_\infty(X)$ and related ideals. \textit{Real Anal. Exchange} \textbf{36} (2010), 45--54.

\bibitem{A} A.R. Aliabad, F. Azarpanah and M. Namdari, Rings of continuous functions vanishing at infinity. \textit{Comment. Math. Univ. Carolin.} \textbf{45} (2004),
    519-–533.

\bibitem{Be} E. Behrends, $M$-structure and the Banach--Stone Theorem. Springer, Berlin, 1979.

\bibitem{Bu} D.K. Burke, Covering properties, in: K. Kunen and J.E. Vaughan (Eds.), Handbook of Set-Theoretic Topology, Elsevier, Amsterdam, 1984, pp. 347--422.

\bibitem{C} W.W. Comfort, On the Hewitt realcompactification of a product space. \textit{Trans. Amer. Math. Soc.} \textbf{131} (1968), 107--118.

\bibitem{E} R. Engelking, General Topology. Second edition. Heldermann Verlag, Berlin, 1989.

\bibitem{HJ} A.W. Hager and D.G. Johnson, A note on certain subalgebras of $C(X)$. \textit{Canad. J. Math.} \textbf{20} (1968), 389--393.

\bibitem{GJ} L. Gillman and M. Jerison, Rings of Continuous Functions. Springer--Verlag, New York--Heidelberg, 1976.

\bibitem{G} C. Good, The Lindel\"{o}f property, in: K.P. Hart, J. Nagata and J.E. Vaughan (Eds.), Encyclopedia of General Topology,  Elsevier, Amsterdam, 2004, pp. 182--184.

\bibitem{Ho} R.E. Hodel, Jr., Cardinal functions I, in: K. Kunen and J.E. Vaughan (Eds.), Handbook of Set-Theoretic Topology, Elsevier, Amsterdam, 1984, pp. 1--61.

\bibitem{Ko3} M.R. Koushesh, Compactification-like extensions. \textit{Dissertationes Math. (Rozprawy Mat.)} \textbf{476} (2011), 88 pp.

\bibitem{Ko4} M.R. Koushesh, The partially ordered set of one-point extensions. \textit{Topology Appl.} \textbf{158} (2011), 509--532.

\bibitem{Ko7} M.R. Koushesh, A pseudocompactification. \textit{Topology Appl.} \textbf{158} (2011), 2191--2197.

\bibitem{Ko6} M.R. Koushesh, The Banach algebra of continuous bounded functions with separable support. \textit{Studia Math.} \textbf{210} (2012), 227--237.

\bibitem{Ko13} M.R. Koushesh, Connectedness modulo a topological property. \textit{Topology Appl.} \textbf{159} (2012), 3417--3425.

\bibitem{Ko11} M.R. Koushesh, Representation theorems for Banach algebras. \textit{Topology Appl.} \textbf{160} (2013), 1781--1793.

\bibitem{Ko12} M.R. Koushesh, Topological extensions with compact remainder. \textit{J. Math. Soc. Japan} (47 pp.) In press.

\bibitem{Ko14} M.R. Koushesh, Continuous mappings with null support. (40 pp.) In preparation.  \verb"arXiv:1302.2235 [math.FA]"

\bibitem{PW} J.R. Porter and R.G. Woods, Extensions and Absolutes of Hausdorff Spaces. Springer--Verlag, New York, 1988.

\bibitem{Steph} R.M. Stephenson, Jr., Initially $\kappa$-compact and related spaces, in: K. Kunen and J.E. Vaughan (Eds.), Handbook of Set-Theoretic Topology, Elsevier, Amsterdam, 1984, pp. 603--632.

\bibitem{T} A. Taherifar, Some generalizations and unifications of $C_K(X)$, $C_\psi(X)$ and $C_\infty(X)$. (13 pp.) \verb"arXiv:1210.6521 [math.GN]"

\bibitem{Va} J.E. Vaughan, Countably compact and sequentially compact spaces, in: K. Kunen and J.E. Vaughan (Eds.), Handbook of Set-Theoretic Topology, Elsevier, Amsterdam, 1984, pp. 569--602.

\bibitem{W} N.M. Warren, Properties of Stone--\v{C}ech compactifications of discrete spaces. \textit{Proc. Amer. Math. Soc.} \textbf{33} (1972), 599-–606.

\bibitem{We} M.D. Weir, Hewitt--Nachbin Spaces. American Elsevier, New York, 1975.

\end{thebibliography}

\end{document}